\documentclass[12pt,leqno]{amsart}

\usepackage[utf8]{inputenc}
\usepackage{amsmath,amsfonts,amsthm,amssymb}

\usepackage[margin=1in]{geometry}

\newtheorem{theorem}{Theorem}[section]

\numberwithin{equation}{section} 
\theoremstyle{definition}

\newenvironment{example}[1][Example.]{\begin{trivlist}
\item[\hskip \labelsep {\bfseries #1}]}{\end{trivlist}}
\newenvironment{remark}[1][\it{Remark.}]{\begin{trivlist}
\item[\hskip \labelsep {\bfseries #1}]}{\end{trivlist}}
\newenvironment{acknowledgments}[1][Acknowledgments.]{\begin{trivlist}
\item[\hskip \labelsep {\bfseries #1}]}{\end{trivlist}}

\title{Kummer Congruences Arising from the Mirror Symmetry of an Elliptic Curve}
\date{\today}
\author{Adele Lopez}
\address{Department of Mathematics and Computer Science, Emory University, Atlanta, Georgia 30322}
\email{adele.lopez@emory.edu}
\subjclass[2010]{Primary: 11F30, 11F33; Secondary: 14J33}
\thanks{The author thanks the NSF for their generous support.}

\begin{document}

\begin{abstract}In the genus 1 case, mirror symmetry reduces to the statement that a certain family of generating functions, relating to an elliptic curve, are quasimodular. In their proof of this fact, Kaneko and Zagier used a related family of generating functions $A_n(\tau)$, which they show to be quasimodular. We show that these $A_n$'s also satisfy Kummer-type congruences. Additionally, we show that for a prime $p$, the $p$th power coefficients of $A_n$ $p$-adically converge to zero, for specific values of $n$. 
\end{abstract}

\maketitle
 
\section{Introduction}

	Mirror symmetry is a special relation between Calabi-Yau manifolds. It concerns the generating function \begin{equation}F_g(t)=\sum_d N_{g,d}q^d,\end{equation} which counts the (suitably adjusted) numbers $N_{g,d}$ of holomorphic curves of genus $g$ and degree $d$ on a Calabi-Yau manifold $X$. These $F_g(t)$ are symplectic manifold invariants of $X$. Mirror symmetry conjectures that these $F_g(t)$ have an alternative interpretation as complex manifold invariants of a family of `mirror' Calabi-Yau manifolds $Y_t$. This has been proven in the case where the Calabi-Yau manifolds have genus 1, where both $X$ and the $Y_t$'s are elliptic curves. In this case, it reduces to the statement that these $F_g(t)$ have a modularity property, specifically, that they are quasimodular~\cite{Dijkgraaf}. 

	The space of quasimodular forms $\widetilde{M}_*(\Gamma_1)$ of level 1 is the graded ring generated by the Eisenstein series $E_2$, $E_4$ and $E_6$, and graded by assigning weight $k$ to $E_k$. If we replace $E_2(\tau)$ (which is not modular), with $E_2^*(\tau)=E_2(\tau) -\frac{3}{\pi\text{Im}(\tau)}$ (which is modular, but not holomorphic), we obtain a weak harmonic Maass form, of which the quasimodular form is the holomorphic part. 
	Kaneko and Zagier proved~\cite{KZ} that the generating functions arising from the genus one case of mirror symmetry are quasimodular, along with a related set of generating functions, by considering the following generalization of the Jacobi theta function:
\begin{equation}
\label{Theta}
	\Theta(X,q,\zeta)=\prod_{m>0}(1-q^m)\prod_{\substack{n>0\\ n \text{ odd}}}(1-e^{n^2X/8}q^{n/2}\zeta)(1-e^{-n^2X/8}q^{n/2}\zeta^{-1})
\end{equation}
as a formal power series in $X$ and $q^{1/2}$ with coefficients in $\mathbb Q[\zeta,\zeta^{-1}]$. A family of quasimodular forms, $A_n(q)$, may now be derived from this. Let $\Theta_0(X,q)$ be the coefficient of $\zeta^0$ in $\Theta(X,q,\zeta)$ considered as a Laurent series in $\zeta$. Now expand $\Theta_0(X,q)$ as a Taylor series in $X$: 
\begin{equation}\label{Theta0}\Theta_0(X,q)=\sum_{n=0}^{\infty}A_n(q)X^{2n}.\end{equation} These $A_n(q)$'s are our objects of interest. They are quasimodular forms of weight $6n$ and level 1, as proved by Kaneko and	Zagier in their paper. This fact can then be used to prove that the $F_g(\tau)$'s, which are the coefficients of $X^{2g-2}$ in $\log\Theta_0$, are quasimodular~\cite{KZ}. (We can think of these $A_n$'s as `counting' the number of not necessarily connected covers of an elliptic curve $E$, with Euler number $2-2g$ and degree $d$.)
Here we prove that the ${A}_n$'s (once normalized to have coprime integer coefficients) also satisfy Kummer-type congruences. \begin{example}For instance, letting $\overline{A}_n$ denote the normalized $A_n$, we have that	

 \begin{align*}
  \overline{A}_2(q)&=q^2 + 80q^3 + 1230q^4 + 9248q^5 + 46020q^6 + 174624q^7  + 549704q^8 + \cdots \\
&\equiv q^2 +  3q^5 + 4q^7  + 3q^{10} + 3q^{12} + 2q^{13} + q^{15} + q^{20} + 4q^{21} + 2q^{23} + \cdots  &\pmod 5\\
 \intertext{ and }
  \overline{A}_4(q)&=q^2 + 6560q^3 + 1673310q^4 + 98704448q^5 + 2504270340q^6 + \cdots\\
&\equiv  q^2 +  3q^5 + 4q^7  + 3q^{10} + 3q^{12} + 2q^{13} + q^{15} + q^{20} + 4q^{21} + 2q^{23} +\cdots &\pmod 5 \\
\intertext{ are congruent modulo 5, and }
\overline{A}_3(q)&=q^2 + 728q^3 + 45990q^4 + 968240q^5 + 10876740q^6 + 81037296q^7  +\cdots \\
&\equiv q^2 + q^8 + 2q^5 + 6q^6 +5q^8 +3q^9 +q^{11} + 5q^{12}+ 2q^{14} + 3q^{15} + \cdots  &\pmod 9 \\
\intertext{and} 
\overline{A}_6(q)&=q^2 + 531440q^3 + 2176254990q^4 + 998066826848q^5 +   \cdots\\
&\equiv q^2 + q^8 + 2q^5 + 6q^6 +5q^8 +3q^9 +q^{11} + 5q^{12}+ 2q^{14} + 3q^{15} + \cdots & \pmod 9 
\end{align*} are congruent modulo 9.
\end{example} In general, we have the following theorem.
\begin{theorem}
\label{kummer}
For a prime $p$ and a positive integer $s$, if $2i\equiv 2j \pmod{ \phi(p^s)}$, then for $n\ge 1$, we have
$$\overline{A}_i(q)\equiv \overline{A}_j(q) \pmod {p^s},$$ where $\phi$ is Euler's totient function. 
\end{theorem}

\begin{remark}
These congruences are not a trivial consequence of the Kummer-type congruences that are well known to hold for the Eisenstein series. We will demonstrate this in the last section.
\end{remark}

\begin{remark}
The author was unable to find a conceptual proof of this theorem from the theory of elliptic curves and mirror symmetry.
\end{remark}

In addition to the relative congruences relating these quasimodular forms, the $p$th power coefficients tend $p$-adically to zero. This follows from Serre's theory of $p$-adic modular forms~\cite{Serre}. More precisely, we have the following theorem.

\begin{theorem}
\label{padic}
If $k\ge 1$ and $p$ is a prime less than $7$ or $6k\equiv 4, 6, 8, 10, 14 \pmod{p-1}$, then $p$-adically, we have 
\[\lim_{n\to\infty} a_k(p^n)=0,\] 
where $a_k(i)$ is the $i$th coefficient of $\overline{A}_k$. 
\end{theorem}

\begin{example}
The $3$rd power coefficients for $\overline{A}_3$ begin with:
\begin{align*}
   a_3(3^1) &= 728  & =&\;\; 2^3\cdot 7 \cdot 13,\\ 
   a_3(3^2) &= 2028730080 & =&\;\; 2^5\cdot 3 \cdot 5 \cdot 13^2 \cdot 89 \cdot 281,\\  
   a_3(3^3) &= 1747100845087920 & =& \;\;2^4\cdot 3^3 \cdot 5 \cdot 7 \cdot 617 \cdot 187275523,\\  
   a_3(3^4) &= 1249380857829754167840 & =&\;\; 2^5\cdot 3^5 \cdot 5 \cdot 7^2 \cdot 794953 \cdot 824956519.\\  
\end{align*}

The increasing powers of 3 illustrate the fact that  $\lim_{n\to\infty} a_3(3^n)=0.$
\end{example}

\begin{acknowledgments} The author thanks Ken Ono for suggesting this problem, and for his helpful advice. 
\end{acknowledgments} 
	
\section{Proofs of Theorems~\ref{kummer} and~\ref{padic}}

Here we prove Theorems~\ref{kummer} and~\ref{padic}.

\begin{proof}[Proof of Theorem~\ref{kummer}]

This result is a consequence of equation (5.9) in Dijkgraaf~\cite{Dijkgraaf}.  For completeness, we will provide the full argument here.

We'll begin by considering the second product from our expression of $\Theta(X,q,\zeta)$ in~(\ref{Theta}). First, replace $n$ with $r=n/2$, and factor it into two infinite products. This product becomes \begin{equation}\prod_{r\in I}(1-e^{r^2X/2}q^{r}\zeta)\prod_{s\in I}(1-e^{-s^2X/2}q^{s}\zeta^{-1}),\end{equation} where $I$ is the set of positive half-integers.

This product has a combinatorial interpretation. A term of this double product can be computed by taking two finite subsets $A, B$ of $I$, and restricting the product to $r\in A$ and $s\in B$. For our purposes, we only need to consider the terms which are constant with respect to $\zeta$, as these are the only terms used in the definition of the $A_n$'s. These terms are precisely the ones for which $|A|=|B|$. Also, note that the exponent of $q$ in the term is counting the combined sum of all the elements of $A$ and $B$. 

Consider pairs of sets $A,B\subset I$ with the same cardinality and the same combined sum $d$. We can construct a unique partition of $d$ from each distinct pair of sets, and vice-versa. Thus, we can arrange the terms of interest into a sum over partitions as follows: 
\begin{equation}\sum_{d=0}^\infty \sum_{\ell\in \text{Partitions}(d)} e^{\lambda(\ell)X}q^d,\end{equation} where 
\[ \lambda(\ell)=\frac{1}{2}\left(\sum_{r\in A}r^2-\sum_{s\in B}s^2\right),\] with $A$ and $B$ being the sets of half-integers corresponding to the partition $\ell$. 

Now, let's relate all this back to the $A_n$'s. From the above reasoning, we can see that $\Theta_0(X,q)$ from (\ref{Theta0}) becomes 
\begin{equation}\Theta_0(X,q)=\prod_{m>0}(1-q^m)\sum_{d=0}^\infty \sum_{\ell\in \text{Partitions}(d)} e^{\lambda(\ell)X}q^d.\end{equation}
Expand the exponential as a Taylor series to get \begin{equation}\Theta_0(X,q)=\prod_{m>0}(1-q^m)\sum_{d=0}^{\infty}\sum_{\ell\in \text{Partitions}(d)} \sum_{k=0}^{\infty}\frac{1}{k!}\lambda(\ell)^k X^k q^d.\end{equation} Hence, we have that 
\begin{equation}A_n(\tau)=\frac{1}{(2n)!}\prod_{m>0}(1-q^m)\sum_{d=0}^{\infty}\sum_{\ell\in \text{Partitions}(d)} \lambda(\ell)^{2n}q^d.\end{equation}

Since for any odd integer $a$, $a^2\equiv 1 \pmod 8$, $\lambda(\ell)$ is always an integer. Considering pairs of a partition and its conjugate, we see that the sum of the $\lambda(\ell)^k$'s for the partitions of a given number will either be zero for odd $k$, and even for even $k$.  Thus, the normalized $A_n$'s are described by  
\begin{equation}\overline{A}_n(\tau)=\frac{1}{2}\prod_{m>0}(1-q^m)\sum_{d=0}^{\infty}\sum_{\ell\in \text{Partitions}(d)}\lambda(\ell)^{2n}q^d.\end{equation} It is simple to check that the coefficient of $q^2$ in $\overline{A}_n$ is one for every $n>0$. 

The theorem then follows from Euler's generalization of Fermat's little theorem. 
 
\end{proof}

The second theorem is a simple consequence of Serre's work on $p$-adic modular forms~\cite{Serre}. 
\begin{proof}[Proof of Theorem~\ref{padic}]

Since quasimodular forms are generated by $E_2$, $E_4$, and $E_6$, and since $E_2$ is a $p$-adic modular form of weight 2~\cite{Serre}, it follows from Kaneko and Zagier~\cite{KZ} that $A_k$ is a $p$-adic modular form of weight $6k$. The theorem then follows from Theorem 7 in Serre~\cite{Serre}. 

\end{proof}

We now will demonstrate the fact that Theorem~\ref{kummer} is not a trivial consequence of the Kummer congruences for the Eisenstein series.

\begin{example}
 Recall that $\overline{A}_2\equiv\overline{A}_4\pmod 5.$ We will convert this into a relation of Eisenstein series. Since the quasimodular forms of level 1 are generated by the Eisenstein series $E_2$, $E_4$, and $E_6$, we can compute 
\begin{align*}
\overline{A}_2(q)&=\frac{-875 E_2^6 + 2220 E_2^4 E_4 - 1791 E_2^2 E_4^2 + 1050 E_4^3 +
 580 E_2^3 E_6 - 1788 E_2 E_4 E_6 + 604 E_6^2}{447897600,}
\intertext{and}
\overline{A}_4(q)&= (-7072690625 E_2^{12} + 29791020000 E_2^{10} E_4 - 17984909250 E_2^8 E_4^2 
  - 41175027180 E_2^6 E_4^3\\ & - 73855453833 E_2^4 E_4^4 +  692323272900 E_2^2 E_4^5 
  + 41478466500 E_4^6 + 33993155000 E_2^9 E_6\\& - 298920573000 E_2^7 E_4 E_6
  + 920662991640 E_2^5 E_4^2 E_6 - 1574832872088 E_2^3 E_4^3 E_6\\& - 887970913200 E_2 E_4^4 E_6
   + 70320075000 E_2^6 E_6^2 -  283741244640 E_2^4 E_4 E_6^2\\& + 1785182642712 E_2^2 E_4^2 E_6^2 
   + 189036658800 E_4^3 E_6^2 - 189291716320 E_2^3 E_6^3\\& - 405118626528 E_2 E_4 E_6^3 
   +17175744112 E_6^4)/60183678025728000.
   \end{align*} If we subtract these, renormalize, and reduce the coefficients of the polynomial in $E_2, E_4$, and $E_6$ modulo 5, we get the following:
 
\[2 (E_2^4 E_4^4 + E_2^3 E_4^3 E_6 + E_2^2 E_4^2 E_6^2 + E_2 E_4 E_6^3 + E_6^4)\equiv 0 \pmod 5.\] This congruence follows from the fact that $E_4\equiv 1 \pmod 5$ and $E_6\equiv E_2 \pmod 5$. However, it is not obvious that $\overline{A}_2 \equiv  \overline{A}_4 \pmod 5$ \it{a priori}.
\end{example}

\bibliography{Bibliography}{}
\bibliographystyle{amsalpha}

\end{document}